\documentclass{article}
\usepackage{amsmath}
\usepackage{amsfonts,amsthm}
\usepackage{amssymb,enumerate,verbatim,enumitem}
\usepackage{graphicx}

\usepackage{hyperref}

\newcounter{capitalcounter}

\newcounter{claimcounter}

\newcounter{myfootnote}[page]

\newtheorem{lemma}{Lemma}[section]
\newtheorem{corollary}[lemma]{Corollary}
\newtheorem{theorem}[lemma]{Theorem}

\newtheorem{prop}[lemma]{Proposition}

\theoremstyle{definition}
\newtheorem{defn}[lemma]{Definition}

\newtheorem{claim}[lemma]{Claim}

\theoremstyle{remark}
\newtheorem*{rmk}{Remark}

\newtheorem*{example}{Example}
\newtheorem*{nte}{Note}

\global\long\def\a{\alpha}

\global\long\def\d{\delta}

\global\long\def\e{\varepsilon}

\global\long\def\N{\mathbb{N}}

\global\long\def\P{\mathbb{P}}

\global\long\def\GG{G}
\global\long\def\HH{\mathcal{H}}

\global\long\def\E{\mathbb{E}}

\global\long\def\re{\begin{rmk}}
\global\long\def\mark{\end{rmk}}
\global\long\def\ex{\begin{example}}
\global\long\def\ple{\end{example}}
\global\long\def\no{\begin{nte}}
\global\long\def\ted{\end{nte}}
\global\long\def\en{\begin{compactenum}}
\global\long\def\um{\end{compactenum}}
\global\long\def\li{\begin{compactitem}}
\global\long\def\st{\end{compactitem}}
\global\long\def\de{\begin{defn}}
\global\long\def\fn{\end{defn}}
\global\long\def\cor{\begin{corollary}}
\global\long\def\ary{\end{corollary}}
\global\long\def\lem{\begin{lemma}}
\global\long\def\ma{\end{lemma}}
\global\long\def\arr{\begin{array}}
\global\long\def\ay{\end{array}}
\global\long\def\pr{\begin{proof}}
\global\long\def\oof{\end{proof}}

\global\long\def\barr#1{\overline{#1}}

\title{Hamiltonicity in random graphs is born resilient}

\author{Richard Montgomery\footnote{University of Birmingham, Birmingham, B15 2TT, UK; r.h.montgomery@bham.ac.uk}}

\date{}

\begin{document}

\maketitle
\begin{abstract}
Let $\{G_M\}_{M\geq 0}$ be the random graph process, where $G_0$ is the empty graph on $n$ vertices and subsequent graphs in the sequence are obtained by adding a new edge uniformly at random. For each $\varepsilon>0$, we show that, almost surely, any graph $G_M$ with minimum degree at least 2 is not only Hamiltonian (as shown independently by Bollob\'as, and Ajtai, Koml\'os and Szemer\'edi), but remains Hamiltonian despite the removal of any set of edges, as long as at most $(1/2-\varepsilon)$ of the edges incident to each vertex are removed. We say that such a graph is \emph{$(1/2-\varepsilon)$-resiliently Hamiltonian}. Furthermore, for each $\e>0$, we show that, almost surely, each graph $G_M$ is not $(1/2+\varepsilon)$-resiliently Hamiltonian.
These results strengthen those by Lee and Sudakov on the likely resilience of Hamiltonicity in the binomial random graph.

For each $k$, we denote by $G^{(k)}$ the (possibly empty) maximal subgraph with minimum degree at least $k$ of a graph $G$. That is, the $k$-\emph{core} of $G$. Krivelevich, Lubetzky and Sudakov have shown that, for each $k\geq 15$, in almost every random graph process $\{G_M\}_{M\geq 0}$, every non-empty $k$-core is Hamiltonian.
 We show that, for each $\varepsilon>0$ and $k\geq k_0(\varepsilon)$, in almost every random graph process $\{G_M\}_{M\geq 0}$, every non-empty $k$-core is $(1/2-\varepsilon)$-resiliently Hamiltonian, but not $(1/2+\varepsilon)$-resiliently Hamiltonian.

Results on resilience of Hamiltonicity in the random graph process have been independently shown by Nenadov, Steger and Truji\'c.
\end{abstract}

\section{Introduction}\label{1introhamres}
The appearance of Hamilton cycles in random graphs has been studied since the pioneering work by Erd\H{o}s and R\'enyi in 1959~\cite{ER59}. As Hamilton cycles, by definition, contain every vertex in their parent graph, any graph containing a Hamilton cycle has minimum degree at least 2. As is well-known, if $p= (\log n+\log\log n-\omega(1))/n$, then the binomial random graph $G(n,p)$ almost surely has a vertex with degree at most 1, and hence is not Hamiltonian. Conversely, improving on breakthrough results by P\'osa~\cite{posa76} and Korshunov~\cite{kor76}, in 1983 Bollob\'as~\cite{bollo83} and Koml\'os and Szemer\'edi~\cite{KS83} independently showed that,
if $p= (\log n+\log\log n+\omega(1))/n$, then $G(n,p)$ is almost surely Hamiltonian.

Suppose instead we consider the $n$-vertex random graph process $\{G_M\}_{M\geq 0}$, where $G_0$ is the graph with $n$ vertices and no edges, and each subsequent graph $G_M$ is formed from $G_{M-1}$ by the addition of an edge uniformly at random, until the complete graph $G_{\binom{n}{2}}$ is formed. Independently, Bollob\'as~\cite{bollo84} and Ajtai, Koml\'os and Szemer\'edi~\cite{AKS85} showed that, in almost every random graph process, the very edge which is added to raise the minimum degree to 2 will also create a Hamilton cycle. The random graph process is strongly linked to binomial random graphs (see, for example, Section~\ref{modelswitch}), and thus we can infer from this beautiful result the previous known results on the likely Hamiltonicity of $G(n,p)$. The result by Bollob\'as demonstrates that the likely obstacle to the existence of a Hamilton cycle in a random graph is the existence of some vertex of degree less than 2. In this paper, we will show that, once this obstacle is overcome, it is very likely that the graph is not only Hamiltonian, but \emph{resiliently} Hamiltonian.

The general study of the resilience of different graph properties in the random graph was initiated by Sudakov and Vu~\cite{SV08}, and has since seen the consideration of a variety of different properties (see, for example,~\cite{ABET16,BCW11,DKMS08}). Given a graph $G$ satisfying a property $\mathcal{P}$, the \emph{local resilience} of $G$ with respect to $\mathcal{P}$ is the largest integer $r$ such that, given any graph $H\subset G$ with maximum degree at most $r$, the graph $G-H$ has the property $\mathcal{P}$. Sudakov and Vu~\cite{SV08}
 showed that, if $p>\log^4 n/n$, then the local resilience of Hamiltonicity in $\GG(n,p)$ is almost surely $(1/2+o(1))pn$, and conjectured that this remains true as long as $p=\omega(\log n/n)$. When $p=\omega(\log n/n)$, Frieze and Krivelevich~\cite{FK08} and Ben-Shimon, Krivelevich and Sudakov~\cite{BSKS1,BSKS2} gave increasingly strong bounds for such likely local resilience of Hamiltonicity in $G(n,p)$, before Lee and Sudakov~\cite{LS12} confirmed this conjecture. More precisely, Lee and Sudakov showed that, for every $\e>0$, there exists some constant $C$, such that, if $p\geq C\log n/ n$, then the local resilience of $\GG(n,p)$ is almost surely between $(1/2-\e)pn$ and $(1/2+\e)pn$.

What hope have we of improving the known range of probability $p$ for which $G(n,p)$ is likely to be resiliently Hamiltonian? As we decrease $p$ within the range of likely Hamiltonicity, the degree sequence of $G(n,p)$ typically becomes increasingly irregular. In particular, the likely minimum degree of $G(n,p)$ will drift proportionally away from the expected average degree $(n-1)p$. For example, when $p=(\log n+2\log\log n)/n$, the minimum degree of $G(n,p)$ is almost surely $2$ or $3$; the likely local resilience of Hamiltonicity cannot then be more than 1. The definition of local resilience here is rather weak, considering few edge sets for removal. In order to better study the resilience of Hamiltonicity in $G(n,p)$ for such values of $p$, the following definition is thus preferable (see also~\cite{BSKS2}).
\begin{defn}
We say a graph $G$ is \emph{$\a$-resilient with respect to the property $\mathcal{P}$} if, given any subgraph $H\subset G$, with $d_H(v)\leq \a d_G(v)$ for each $v\in V(G)$, the graph $G-H$ has property~$\mathcal{P}$.
\end{defn}
If $p=\omega(\log n/n)$, then it is very likely that each vertex in $G(n,p)$ has degree $(1+o(1))pn$. Therefore, the result by Lee and Sudakov~\cite{LS12} quoted above implies that, if $p=\omega(\log n/n)$, then $\GG(n,p)$ is almost surely $(1/2-o(1))$-resiliently Hamiltonian. In this paper, we extend this by proving that, in almost every random graph process, every Hamiltonian graph is $(1/2-o(1))$-resiliently Hamiltonian.

\begin{theorem}\label{hamres}
Let $\e>0$. In almost every $n$-vertex random graph process $\{G_M\}_{M\geq 0}$, the following is true for each $0\leq M\leq \binom{n}{2}$. If $\delta(G_M)\geq 2$, then $G_M$ is $(1/2-\e)$-resiliently Hamiltonian, but not $(1/2+\e)$-resiliently Hamiltonian.
\end{theorem}

Ben-Shimon, Krivelevich and Sudakov~\cite{BSKS2} used a further, more general, definition of resilience when studying the Hamiltonicity of random graphs. Let $\mathbf{k}$ be a sequence of $n$ integers. A graph $G$ with the vertex set $[n]$ is said to be $\mathbf{k}$-resilient with respect to the property $\mathcal{P}$ if, for any subgraph $H\subset G$ with $d_H(i)\leq \mathbf{k}(i)$, for each $i$, the graph $G-H$ has the property $\mathcal{P}$. Ben-Shimon, Krivelevich and Sudakov showed that, for every $\e>0$ and $p= (\log n+\log\log n+\omega(1))/n$, the random graph $G=\GG(n,p)$ is almost surely $\mathbf{k}$-resiliently Hamiltonian with
\[
\mathbf{k}(i)=\left\{\begin{array}{ll}
d_G(i)-2 & \text{ if }d_G(i)\leq pn/100, \\
(1/3-\e)d_G(i) & \text{ otherwise}. \\
\end{array}\right.
\]
The methods of this paper could be used to extend this result to use $\mathbf{k}(i)=d_G(i)-2$ if $d_G(i)\leq pn/100$, and $\mathbf{k}(i)=(1/2-\e)d_G(i)$ otherwise; the constant $1/2$ would then be tight. We will concentrate, however, on proving the cleaner statement of Theorem~\ref{hamres}.

For large $k$, we will also use our methods to demonstrate that, in almost every random graph process, the \emph{$k$-core} is born \emph{resiliently} Hamiltonian. From the results mentioned above, we know that the possible obstruction to Hamiltonicity in any random graph $G(n,p)$ is almost surely vertices of degree 0 or 1. If we iteratively remove vertices of degree 0 or 1, are we likely to find a Hamiltonian subgraph? This process would find the largest subgraph with minimum degree at least 2, a structure known as the $2$-core.
For more general $k$, the $k$-core of a graph $G$, denoted $G^{(k)}$, is the (possibly empty) maximal subgraph with minimum degree at least $k$. This concept was introduced by Bollob\'as~\cite{bollo84}, who showed that, for each $k\geq 3$, there exists a constant $C(k)$ such that, if $p\geq C(k)/n$, then the $k$-core of $G(n,p)$ is almost surely non-empty and $k$-connected. Among results for more general $k$, \L uczak~\cite{Luc87} showed that, if $p=(\log n+6\log\log n+\omega(1))/3n$, then the 2-core of $G(n,p)$ is almost surely Hamiltonian.
In other words, well before the random graph is reliably Hamiltonian, it is likely that, if we iteratively remove vertices with degree 0 or 1, then the remaining graph is Hamiltonian.

For each $k\geq 3$, Bollob\'as, Cooper, Fenner and Frieze~\cite{BCFF00} showed that there is some $C(k)$ for which, if $p\geq C(k)/n$, then the $k$-core of $G(n,p)$ is almost surely Hamiltonian, where $C(k)=(2+o_k(1))k^3$. For each $k\geq 15$, Krivelevich, Lubetzky and Sudakov~\cite{KLS14} subsequently showed that, in almost every random graph process, every non-empty $k$-core is Hamiltonian. We will show that, in almost every random graph process, every non-empty $k$-core is furthermore $(1/2-o_k(1))$-resiliently Hamiltonian, as follows.

\begin{theorem}\label{coreres} For each $\e>0$, there exists some $k_0$ such that, for each $k\geq k_0$, in almost every $n$-vertex random graph process $\{G_M\}_{M\geq 0}$, the following holds for each $0\leq M\leq \binom{n}{2}$. If $G_M^{(k)}\neq \emptyset$, then $G_M^{(k)}$ is $(1/2-\e)$-resiliently Hamiltonian, but not $(1/2+\e)$-resiliently Hamiltonian.
\end{theorem}

Our methods to prove Theorem~\ref{coreres} are quite different from those used by Krivelevich, Lubetzky and Sudakov in~\cite{KLS14}, and owe more to the work by Lee and Sudakov~\cite{LS12} on the resilience of Hamiltonicity in $G(n,p)$. The techniques we use offer an alternative proof of the result by Krivelevich, Lubetzky and Sudakov~\cite{KLS14} that, for large $k$, in almost every random graph process, when the $k$-core is non-empty it is Hamiltonian. This alternative proof would be simpler than that found in~\cite{KLS14}, but the techniques only work for larger values of $k$ (perhaps, with appropriate optimization, for $k$ slightly below 100) rather than for all $k\geq 15$ as in~\cite{KLS14}.

The heart of our paper is in the proof of a more general result, Theorem~\ref{newtheorem}, which provides an (almost-sure) rule that determines certain induced large subgraphs of $G(n,p)$ are resiliently Hamiltonian. It is then relatively simple to confirm that the subgraphs in Theorem~\ref{hamres} and~\ref{coreres} almost surely satisfy this rule. In order to state Theorem~\ref{newtheorem}, we require the following two definitions.

\begin{defn}
We say $H$ is an \emph{$\alpha$-residual} subgraph of a graph $G$ if, for each $v\in V(H)$, we have $d_H(v)\geq \a d_{G[V(H)]}(v)$.
\end{defn}

\de\label{expdefn} A graph $H$ is a \emph{$2$-expander} if it is connected and, for every subset $U\subset V(H)$ with $|U|\leq |H|/8$, we have $|N(U)|\geq 2|U|$.
\fn

\begin{theorem}\label{newtheorem}
For each $\e>0$, there exists $\d,C>0$ such that, if $p\geq C/n$, then $G=G(n,p)$ has the following property with probability $1-o(n^{-3})$. If $H$ is a $(1/2+\e)$-residual subgraph of $G$, with $|H|\geq \e n$, which contains a spanning 2-expander with at most $\delta pn^2$ edges, then $H$ is Hamiltonian.
\end{theorem}

The constant $1/2$ appearing in Theorem~\ref{newtheorem} cannot be reduced. We confirm this with the following lemma.

\begin{lemma}\label{nonresil}
For each $\e>0$, there exists $C$ such that, if $p\geq C/n$, then, with probability $1-o(n^{-3})$, $G(n,p)$ contains no subset $U\subset V(G)$, with $|U|\geq \e n$, for which $G[U]$ is $(1/2+\e)$-resiliently Hamiltonian.
\end{lemma}

In Section~\ref{secprelim}, we recall some simple properties of random graphs. In Section~\ref{secnewthm}, we sketch the proof of Theorem~\ref{newtheorem}, before proving it in detail. Lemma~\ref{nonresil} is proved in Section~\ref{secnonres}. For Theorems~\ref{hamres} and~\ref{coreres}, we demonstrate that large sets resiliently expand in the same manner, proving Theorem~\ref{newnewtheorem} from Theorem~\ref{newtheorem} in Section~\ref{secresil}. Using Theorem~\ref{newnewtheorem} and Lemma~\ref{nonresil}, Theorems~\ref{hamres} and~\ref{coreres} are proved in Sections~\ref{secfull} and~\ref{seccore}, respectively. In the rest of this section, we will cover our basic notation.

\medskip
\noindent\textbf{Remark.} Theorem~\ref{hamres} has also been shown in independent work by Nenadov, Steger and Truji\'c~\cite{NST18}. In addition, they obtain a similar result for matchings in the random graph process, and show that in almost every random graph process $\{G_M\}_{M\geq 0}$, for each $M\geq (1/6+o(1))n\log n$ the 2-core of $G_M$ is $(1/2-o(1))$-resiliently Hamiltonian. These results could also be obtained using similar methods to those used in this paper. The methods in~\cite{NST18}, however, are significantly different from those used here, and in particular do not work when $M=o(n\log n/\log\log n)$ (c.f.\ Theorem~\ref{coreres}) without, at the least, significant extra analysis (see also the discussion at the end of Section~\ref{secsketch}).


\subsection{Notation}
A graph $G$ has vertex set $V(G)$, edge set $E(G)$, minimum degree $\delta(G)$ and maximum degree $\Delta(G)$, and $|G|=|V(G)|$. When $A\subset V(G)$, $N(A)$ is the set of neighbours of vertices in $A$ in $V(G)\setminus A$. When $x\in V(G)$, $d(x)$ is the degree of $x$ in $G$. Where multiple graphs are considered, we refer to the relevant graph in the subscript, using, for example, $d_G(x)$.

For a graph $G$ and a vertex set $U\subset V(G)$, $G[U]$ and $G-U$ are the induced subgraphs of $G$ with vertex sets $U$ and $V(G)\setminus U$, respectively. For any graphs $G$ and $H$, $G-H$ is the graph with vertex set $V(G)$ and edge set $E(G)\setminus E(H)$, and $G\cup H$ is the graph with vertex set $V(G)\cup V(H)$ and edge set $E(G)\cup E(H)$. For a graph $G$ and a set $E\subset \binom{V(G)}{2}$, $G+E$ and $G-E$ are the graphs with vertex set $V(G)$ and edge sets $E(G)\cup E$ and $E(G)\setminus E$, respectively. For any graph $G$ and vertex sets $A,B\subset V(G)$, $e_G(A,B)$ is the number of pairs $(x,y)$ with $xy\in E(G)$, $x\in A$ and $y\in B$.

If $f(n)/g(n)\to 0$ as $n\to\infty$, then we say $g(n)=\omega(f(n))$ and $f(n)=o(g(n))$. If there exists a constant $C$ for which $f(n)\leq C g(n)$ for all $n$, then we say $f(n)=O(g(n))$ and $g(n)=\Omega(f(n))$. If $f=O(g(n))$ and $f(n)=\Omega(g(n))$, then we say that $f(n)=\Theta(g(n))$. The binomial random graph $G(n,p)$ has vertex set $[n]=\{1,\ldots,n\}$ and edges chosen independently at random with probability $p$. We denote the complete graph on $[n]$ by $K_n$.


\section{Preliminaries}\label{secprelim}
In this section, we first cover how we move between the binomial random graph and the random graph process, before giving a few simple properties of the binomial random graph.

\subsection{Model switching}\label{modelswitch}
We will often find it convenient to show properties hold in $G(n,p)$, before moving to the random graph process using the following standard lemma (see, for example, Bollob\'as~\cite{bollorand}), where $G_{n,M}$ is chosen uniformly at random from the graphs with vertex set $[n]$ and $M$ edges.


\lem\label{switch} Let $n\in\N$, $1\leq M\leq \binom{n}{2}$ and $p=M/\binom{n}{2}$, and let $\mathcal{P}$ be a graph property. Then
\[
\hspace{3.4cm}\P(\GG_{n,M}\text{ has property }\mathcal{P})\leq 2n\cdot\P(\GG(n,p)\text{ has property }\mathcal{P}).\hspace{3.2cm}\qed
\]
\ma
Typically then, we will show, for any $p$, that a property holds in $\GG(n,p)$ with probability $1-o(n^{-3})$. By Lemma \ref{switch} then, it holds in $\GG_{n,pN}$ with probability $1-o(n^{-2})$. As, in the $n$-vertex random graph process $\{G_M\}_{M\geq 0}$, $G_M$ is distributed as $G_{n,M}$, the property therefore holds throughout almost every random graph process.

\subsection{Properties of the binomial random graph}
The probabilistic results we need for Theorem~\ref{newtheorem} follow simply from Chernoff's inequality (see, for example, Janson, \L uczak and Ruci\'nski~\cite[Corollary 2.3]{jlr11}).
\lem\label{chernoff} If $X$ is a binomial variable with standard parameters~$n$ and $p$, denoted $X=\mathrm{Bin}(n,p)$, and $\e$ satisfies $0<\e\leq 3/2$, then
\[
\hspace{4.8cm}\P(|X-\E X|\geq \e \E X)\leq 2\exp\left(-\e^2\E X/3\right).\hspace{4.1cm}\qed
\]
\ma
When $A$ and $B$ are vertex sets and $G=G(n,p)$, the parameter $e_G(A,B)$ is close to being binomially distributed, and its typical value can be bounded using the following simple proposition.
\begin{prop}\label{Chern} Let $n\in\N$ and $0\leq p\leq 1$. Suppose $X_1,\ldots,X_n$ are independent random variables, each equal to 1 with probability $p$, and 0 otherwise. Suppose $\d_i\in\{1,2\}$, $i\in[n]$, and $X=\sum_{i}\d_iX_i$. Then, for each $0<\e<1$,  we have
\begin{equation}\label{thingy}
\P(|X-\E X|\geq \e\E X)\leq 4\exp(-\e^2\E X/9).
\end{equation}
\end{prop}
\pr Note that, if $\E X\leq 9p$, then the right hand side of~\eqref{thingy} is larger than 1 and the result is trivial. Assume that $\E X\geq 9p$. For each $i\in[n]$, pick $\d_{i,1},\d_{i,2}\in\{0,1\}$ so that $\d_i=\d_{i,1}+\d_{i,2}$ and, if $Y_1=\sum_{i}\d_{i,1}X_{i}$ and $Y_2=\sum_{i}\d_{i,2}X_{i}$, then $|\E Y_1-\E Y_2|\leq p$. As $\E X\geq 9p$, we have $\E Y_1,\E Y_2\geq \E X/3$. By Lemma~\ref{chernoff}, $\P(|Y_1-\E Y_1|\geq \e\E Y_1)\leq 2\exp(-\e^2\E Y_1/3)\leq 2\exp(-\e^2\E X/9)$. A similar result holds for $Y_2$, so that
\[
\P(|X-\E X|\geq \e \E X)\leq \P(|Y_1-\E Y_1|\geq \e\E Y_1)+ \P(|Y_2-\E Y_2|\geq \e\E Y_2)\leq 4\exp(-\e^2\E X/9).\qedhere
\]
\oof
Using Proposition~\ref{Chern}, we can give a simple bound on the number of edges we can expect between any two large sets in $G(n,p)$, as follows.

\lem\label{AKS1equiv} Let $\e>0$ and $G=G(n,p)$. With probability $1-o(n^{-3})$, if $A, B\subset V(G)$ and $p|A||B|\geq 100n/\e^2$, then $(1-\e)p|A||B|\leq e_G(A,B)\leq (1+\e)p|A||B|$.
\ma
\pr Let $F=\{uv:u\in A,v\in B\}$, and, for each $e\in F$, let $X_e$ be the indicator variable for $\{e\in E(G)\}$. For each $uv\in F$, let $\delta_{uv}=2$ if $\{u,v\}\subset A\cap B$, and let $\delta_{uv}=1$ otherwise. Note that $e_G(A,B)=\sum_{e\in F}\delta_eX_e$. By Proposition \ref{Chern}, the property in the lemma does not hold then with probability at most
\[
\sum_{A,B\subset V(G),\,p|A||B|\geq 100n/\e^2}4\exp(-\e^2p|A||B|/9)\leq 2^{2n}\cdot 4\exp(-10n)=o(n^{-3}).\qedhere
\]
\oof
It will be useful to have a bound on the expected number of subgraphs of $G(n,p)$ with at most $\delta pn^2$ edges, for any small fixed $\delta$. For this, we will use the following proposition.
\begin{prop}\label{expectedH} For each $0<\d<1$, there exists $n_0$ such that, for each $n\geq n_0$ and $p\geq 1/n$,
\[
\sum_{H\subset K_n, e(H)\leq \delta pn^2}\P(H\subset G(n,p))\leq \exp(2\d\log(e/\delta)pn^2).
\]
\end{prop}
\begin{proof} We have
\begin{align*}
\sum_{H\subset K_n, e(H)\leq \delta pn^2}\P(H\subset G(n,p))&\leq \sum_{i=0}^{\d pn^2}\binom{n^2}{i}p^i
\leq \sum_{i=0}^{\d pn^2}\left(\frac{epn^2}{i}\right)^i
\leq n^2\cdot \left(e/\delta\right)^{\d pn^2} \\
&\leq \exp(2\d\log(e/\delta)pn^2),
\end{align*}
where the last inequalities hold for sufficiently large $n\geq n_0$.
\end{proof}


\section{Proof of Theorem~\ref{newtheorem}}\label{secnewthm}
We begin this section by reintroducing P\'osa rotation and giving a sketch of our proof of Theorem~\ref{newtheorem}. In Sections~\ref{secrotatemany} and~\ref{secextend} we prove two results, Lemma~\ref{rotateresil} and Lemma~\ref{extend}, which allow us to prove Theorem~\ref{newtheorem} in Section~\ref{secnewthmproof}.

\subsection{P\'osa rotation and proof sketch}~\label{secsketch}
The rotation-extension technique was first introduced by P\'osa~\cite{posa76} to study the threshold for Hamiltonicity in the random graph.
Given a path $x_1x_2\ldots x_k$ and an edge $x_kx_j$, for some $j<k-1$, a \emph{rotation with $x_1$ fixed} is made by \emph{breaking} the edge $x_jx_{j+1}$ and considering the new path $x_1\ldots x_jx_kx_{k-1}\ldots x_{j+1}$. Thus, we find a new path with the same vertex set but which starts at $x_1$ and ends at $x_{j+1}\neq x_k$. If we are able to rotate both ends of the path multiple times, then we can find many pairs of vertices, such that, if any one of the pairs is an edge, then there is a cycle with the same vertex set as the path. 

As shown by P\'osa~\cite{posa76}, if a graph $H$ is a 2-expander, then we can rotate a maximal length path in $H$ many times to find other maximal length paths with a different endvertex, to get the following lemma.

\begin{lemma}\label{posalemma}
If $H$ is a $2$-expander and $U\subset V(H)$ supports a maximal length path in $H$ with endvertex $v\in U$, then there are at least $|H|/8$ vertices $u\in U$ for which there is a $u,v$-Hamilton path in $H[U]$.\hfill\qed 
\end{lemma}

If any of the vertex pairs $u,v\in U$ in Lemma~\ref{posalemma} are added to $E(H)$ then $H$ contains a cycle with length $|U|$. Furthermore, if $H$ is connected and $|U|<|H|$, by considering a neighbour of this cycle we can find a path with $|U|+1$ vertices. We will start with a sparse 2-expander in our random graph and add edges which increase the maximum length of a path in the subgraph, or make it Hamiltonian. To describe this, we use the following definition.

\de In a graph $H$, we say $E\subset \binom{V(H)}{2}$ is a \emph{booster for $H$} if $H+E$ contains a longer path than $H$ does, or $H+E$ is Hamiltonian. If $e\in \binom{V(H)}{2}$, and $\{e\}$ is a booster for $H$, then we also say that $e$ is a booster for $H$.
\fn

Note that, if we iteratively add $|H|$ boosters to $H$, then, as the length of the maximum path is at most $|H|-1$, the resulting graph must be Hamiltonian.

A standard method to find a Hamilton cycle in $G(n,(\log n+\log\log n+\omega(1))/n)$ runs as follows. Letting $G_0=G(n,(\log n+\log\log n+\omega(1))/n)$, we can easily show that $G_0$ is almost surely a 2-expander. Therefore, from Lemma~\ref{posalemma} applied twice to a maximal length path in $G_0$, $\binom{V(G_0)}{2}$ contains at least $n^2/128$ boosters for $G_0$. Revealing more edges with probability $10^3/n^2$ (say) to get $G_1$, with probability at least $1/2$, $G_0$ has some booster in $G_1$. Repeating this $k=\omega(n)$ times to get $G_0\supset G_1\supset \ldots \supset G_k$, we almost surely have at least $n$ values for $i$ for which $G_i-G_{i-1}$ contains a booster for $G_{i-1}$, and hence $G_k$ is Hamiltonian. Note that in total each edge has been revealed with probability $(\log n+\log\log n+\omega(1))/n$. This method is known as \emph{sprinkling}, but cannot withstand the later removal of edges. We need therefore new methods to show that random graphs are \emph{resiliently} Hamiltonian.


From the work by Lee and Sudakov~\cite{LS12} on the resilience of Hamiltonicity in $G(n,\omega(\log n/n))$, we can learn the following principle:  almost surely, if we have any sparse subgraph $H_0\subset G=G(n,\omega(1/n))$ and desire one of $\Omega(n^2)$ possible edges to exist in $G-H_0$, then some such edge does exist in $G-H_0$. In this, the desired edges in $G-H_0$ can be determined by $H_0$. The principle comes, as for each possible sparse subgraph $H_0$, it is far more likely that at least one of the $\Omega(n^2)$ desired edges exists in $G-H_0$ than that $H_0\subset G$, so much so that this can overpower the possible number of sparse subgraphs $H_0$. A calculation along this line appears after Claim~\ref{claimf} in the proof of Lemma~\ref{rotatemany}.

From the work by Lee and Sudakov~\cite{LS12}, with only slight modification we have the following method to find a Hamilton cycle in $G(n,\omega(\log n/n))$. It is easy to show that $G(n,2\log n/n)$ is almost surely a 2-expander, and therefore $G=G(n,\omega(\log n/n))$ almost surely contains a sparse $2$-expander $H_0$. From the principle above and Lemma~\ref{posalemma}, we can show that, almost surely, for any sparse 2-expander $H\subset G$, $G-H$ contains a booster edge for $H$. Starting then, with $H_0$, we can iteratively add booster edges to $H_0$ until it is Hamiltonian, where the graph remains relatively sparse as at most $n$ booster edges are added. Importantly, this basic technique can be made to withstand the later removal of edges. As Lee and Sudakov showed, a careful analysis reveals that, for sparse 2-expanders $H_0$, the possibilities for booster edges in $G-H_0$ are so numerous that almost surely some must lie in $H-H_0$ for any $(1/2+\e)$-residual subgraph $H\subset G$. Furthermore, almost surely, any such subgraph $H\subset G$ can be shown to contain a spanning sparse 2-expander, so that this argument can be used to show that $G(n,\omega(\log n/n))$ is almost surely $(1/2-\e)$-resiliently Hamiltonian.

We have already made a small conceptual change to Lee and Sudakov's argument which helps us use a lower edge probability (namely, that they consider all maximal paths within such a subgraph $H_0$, rather than the subgraph alone).
The main novelty, however, in our methods is how we find booster sets. Previous work in this area, in particular that by Lee and Sudakov, used only booster sets consisting of a single edge. We will use instead booster sets with 2 edges, using an application of the principle outlined above to find each of the 2 edges. This works as follows.

Using the notation above, we first show, with a simple application of P\'osa rotation, that there are many edges $e\notin E(H_0)$ on the vertex set $V(H_0)$ for which $H_0+e$ has many boosters. Using the principle outlined above, as $H_0$ is sparse, it is very likely many of these edges will appear in $G-H_0$ -- so many that plenty must lie in any $(1/2+\e)$-residual subgraph $H\subset G$. In fact, in this case, we can show that there is a subgraph $H_1\subset H-H_0$ which is as sparse as $H_0$ yet contains many of these edges.

In summary, we get a sparse subgraph $H_1$ for which there are many edges $f$ for which there is some $e\in E(H_1)$ so that $\{e,f\}$ is a booster for $H_0$ (see Lemma~\ref{rotateresil}). Applying the above principle again, we can show that is it very likely that many such $f$ lie in $G$ -- so many that plenty must again lie in the $(1/2+\e)$-residual subgraph $H$. This allows us to find a booster set $\{e,f\}\subset H-H_0$. Applying this iteratively, we add 2-edge booster sets from $H-H_0$ to $H_0$ until it becomes Hamiltonian.

We emphasise two points here. First, it is important that we do not add all of $H_1$ to $H_0$, but only one edge, else the density of $H_1$ will grow far too quickly. Secondly, it is important that $H_1$ is also sparse to allow us to apply the principle above again.

In each iteration of the above outline, we use only two rotations using P\'osa's method: one rotation in $H_0$ and then one rotation using an edge of $H_1$. In contrast, in~\cite{LS12}, many rotations were needed to show that there were enough single edge boosters for $H_0$ that one will resiliently exist in $G-H_0$. This analysis crucially uses that vertices there will have degree $\omega(\log n/\log\log n)$, so that an edge probability $p=\omega(\log n/\log \log n)$ in $G(n,p)$ is required. A different analysis might work with a lower probability bound, but this appears to be much more difficult than using only two rotations.

We will show such a sparse subgraph $H_1$ exists in Section~\ref{secrotatemany} (finding many candidates for the first edge of the booster), before showing in Section~\ref{secextend} that, for some edge in $H_1$, there is another edge in $H-H_0-H_1$ to forma booster with it (finding the second edge for the booster). Theorem~\ref{newtheorem} is then proved in Section~\ref{secnewthmproof}.

\subsection{Resilient rotation}\label{secrotatemany}

We will work within a random graph $G=G(n,p)$ with a sparse 2-expander $H_0\subset G$ and a $(1/2+\e)$-residual subgraph $H\subset G$ with $V(H)=V(H_0)$. We aim to find another sparse subgraph $H_1\subset G-H_0$, so that, for many edges $e\in E(H_1)$, $H_0+e$ is rich in boosters for $H_0$. This is made precise, as follows.

\de\label{emany}
 Given two graphs $H_0$ and $H_1$ which are edge disjoint but have the same vertex set, we say $H_0$ has \emph{$\e$-many boosters with help from $H_1$} if there are at least $\e |H_0|$ vertices $v\in V(H_0)$ for which there are at least $(1/2+\e)|H_0|$ many vertices $u\in V(H_0)\setminus \{v\}$ for which there exists an $e\in E(H_0)\cup E(H_1)$ so that $\{uv,e\}$ is a booster for $H_0$.
\fn
Our key result in this section is then the following.

\begin{lemma}\label{rotateresil}
For each $0<\e\leq 1$ there exists $\d,C>0$ such that, if $p\geq C/n$, then $G=G(n,p)$ has the following property with probability $1-o(n^{-3})$. Let $H_0\subset G$ be a $2$-expander with at least $\e n$ vertices and at most $2\delta pn^2$ edges and let $H$ be a $(1/2+\e)$-residual subgraph of $G$ with $V(H)=V(H_0)$. Then, there is some subgraph $H_1\subset H-H_0$ with $e(H_1)\leq 2\delta pn^2$ so that $H_0$ has $(\e/16)$-many boosters with help from $H_1$.
\end{lemma}


In order to show that, almost surely, the subgraph $H_1$ in Lemma~\ref{rotateresil} exists, we will first show that there are a linear number of vertices $v$ in $H_0$ for which there are $(1/2+\e)|H|$ vertices $u\in V(H_0)$ for which we can find a set $E_{v,u}$ such that, for each $e\in E_{v,u}$, $uv$ is a booster for $H_0+e$ (see Lemma~\ref{rotatemany}). Taking a random sparse subgraph $H_1\subset G-H$ typically retains some edge in enough of the sets $E_{v,u}$ to allow us then to show that such a graph $H_1$ is very likely to satisfy Lemma~\ref{rotateresil}.

\begin{lemma}\label{rotatemany}
For each $0<\e\leq 1$ there exists $\d,C>0$ such that, if $p\geq C/n$, then $G=G(n,p)$ has the following property with probability $1-o(n^{-3})$. Let $H_0\subset G$ be a $2$-expander with at least $\e n$ vertices and at most $2\delta pn^2$ edges and let $H$ be a $(1/2+\e)$-residual subgraph of $G$ with $V(H)=V(H_0)$. Then, for at least $|H|/8$ vertices $v\in V(H)$ there is some set $U_v\subset V(H)$ with $|U_v|\geq (1/2+\e/8)|H|$ and disjoint subsets $E_{v,u}\subset E(H-H_0)$, $u\in U_v$, so that $|E_{v,u}|\geq 50/\e\delta$ and, for each $u\in U_v$ and $e\in E_{v,u}$, $\{uv,e\}$ is a booster for $H_0$.
\end{lemma}
\begin{proof}
Let $\delta,C>0$, where we will later require that $\delta$ be taken to be small, depending on $\e$, and then subsequently that $C$ be taken to be large. Let $\HH$ be the set of all 2-expander graphs $H_0$ with $V(H_0)\subset [n]$, $e(H_0)\leq 2\delta pn^2$ and $|H_0|\geq \e n$. Let $H_0\in \HH$. By Lemma~\ref{posalemma}, there are at least $|H_0|/8$ vertices $v$, say those in $W_0$, which appear at the end of a longest path in $H_0$.

For each $v\in W_0$, let $f(H_0,v)$ be the event that, for each $(1/2+\e)$-residual subgraph $H$ of $G=G(n,p)$ with $V(H)=V(H_0)$, there is some set $U\subset V(H)$ with $|U|\geq (1/2+\e/8)|H|$ and disjoint subsets $E_u\subset E(H-H_0)$, $u\in U$, so that $|E_{u}|\geq 50/\e\delta$ and, for each $u\in U$ and $e\in E_{u}$, $\{uv,e\}$ is a booster for $H_0$.

\begin{claim} For each $H_0\in \HH$ and $v\in W_0$, $\P(\barr{f(H_0,v)})\leq \exp(-\e^5 pn^2/10^6)$.\label{claimf}
\end{claim}

From this claim, we can show, as follows, that, with probability $1-o(n^{-3})$, $f(H_0,v)$ holds for each $H_0\in \HH$ and $v\in W_0$ with $H_0\subset G$. Note that the events $f(H_0,v)$ and $H_0\subset G$ are independent. The probability that $f(H_0,v)$ does not hold and $H_0\subset G$ for some $H_0\in \HH$ and $v\in W_0$ is then at most
\begin{align*}
\sum_{H_0\in \HH}\sum_{v\in W_0}\P (\barr{f(H_0,v)})\cdot\P(H_0\subset G)
 &\leq n\cdot \exp(-\e^5 pn^2/10^6)\cdot\sum_{H_0\subset K_n,e(H_0)\leq 2\d pn^2}\P(H_0\subset G)\\
&\leq n\cdot\exp(-(\e^5/10^6-2\d\log(e/2\delta))pn^2),
\end{align*}
where we have used Proposition~\ref{expectedH}.
Therefore, we can choose $\delta$ to be sufficiently small, depending only on $\e$, for this to hold with probability $o(n^{-3})$. Thus, we need only prove Claim~\ref{claimf}.

\begin{proof}[Proof of Claim~\ref{claimf}.] Let $H_0\in \HH$ and $v\in W_0$. Let $k$ be the length of a longest path in $H_0$.
Pick some path $P\subset H_0$ with length~$k$ and $v$ as an end-vertex. By Lemma~\ref{posalemma}, as $H_0$ is a 2-expander, we can find a set $A\subset V(P)$ with $|A|=\e |H_0|/30$ so that, for each $a\in A$, there is a $v,a$-path, $P_a$ say, in $H_0[V(P)]$ with length $k$.

Let $B=V(H_0)\setminus A$. For each $u\in B\cap V(P)$, let $X_{u}$ be the set of pairs $\{a,b\}$ with $a\in A$ and $b\in V(P)\setminus A$ so that~$P_a$ could be rotated in $H_0+ab$ with $v$ fixed, using the edge $ab$, to get $u$ as a new endpoint. Note that, here, $\{uv,ab\}$ is a booster for $H_0$, and the sets $X_u$, $u\in B\cap V(P)$ are disjoint and satisfy $|X_u|\leq |A|$. For each $a\in A$, there are $|P|-|A|$ value of $u\in V(P)\setminus A$ for which $P_a$ can be rotated using $au$. Each such rotation will give a different endvertex, at most $|A|$ of which can be in $A$. Thus, there are at least $|P|-2|A|$ values of $u\in V(P)\setminus A=B\cap V(P)$ for which $a\in X_u$. Therefore, counting over all $a\in A$, we have that $|\cup_{u\in B\cap V(P)}X_{u}|\geq |A|(|P|-2|A|)$.

For each $u\in B\setminus V(P)$, let $X_{u}$ be the set of pairs $\{u,a\}$ with $a\in A$. Note that not only are the sets $X_u$ disjoint from each other, but, as they contain no vertices in $V(P)\setminus A$, they are also disjoint from each set $X_u$, $u\in B\cap V(P)$. Note that, for each $u\in B\setminus V(P)$, $|X_u|= |A|$ and $\{ua\}$, and hence $\{uv,ua\}$, is a booster for $H_0$.  Note that $|\cup_{u\in B\setminus V(P)}X_{u}|= |A|(|H_0|-|P|)$, and hence
\begin{equation}\label{thisotherone}
|\cup_{u\in B}X_u|\geq |A|(|H_0|-|P|)+|A|(|P|-2|A|)=|A||H_0|-2|A|^2.
\end{equation}


For each $u\in B$, let $Y_u=X_u\setminus E(H_0)$.
Then,
\[
|\cup_{u\in B}Y_u|\overset{\eqref{thisotherone}}{\geq} |A||H_0|-2|A|^2-2\delta pn^2\geq (1-\e/8)|A||H_0|,
\]
where the second inequality follows for sufficiently small $\d$. For each $u\in B$, we have, furthermore, that $|Y_u|\leq |X_u|\leq |A|$.
For each $a\in A$, let $Z_a$ be the pairs in $\cup_{u\in B}Y_u$ which contain $a$, so that $\cup_{a\in A}Z_a=\cup_{u\in B}Y_u$, and, hence, $|\cup_{a\in A}Z_a|=|\cup_{u\in B}Y_u|\geq (1-\e/8)|A||H_0|$.

Define the events $F_1$, $F_2$, and $F_3$, as follows.
\begin{enumerate}[label = $F_1$:]
\item $|\cup_{a\in A}(Z_a\cap E(G))|\geq (1-\e/4)p|A||H_0|$.
\end{enumerate}
\begin{enumerate}[label = $F_2$:]
\item $\sum_{a\in A}d_{G[V(H_0)]-H_0}(a)\leq (1+\e/8)p|A||H_0|$.
\end{enumerate}
\begin{enumerate}[label = $F_3$:]
\item If $U\subset B$ with $|U|\leq (1/2+\e/8)|H_0|$, then $|\cup_{u\in U}(Y_u\cap E(G))|< (1/2+\e/4)p|A||H_0|$.
\end{enumerate}

As $|A||H_0|\geq \e^3n^2/30$, and $\E|\cup_{a\in A}(Z_a\cap E(G))|\geq (1-\e/8)p|A||H_0|$, by Lemma~\ref{chernoff}, we have $\P(\barr{F_1})\leq \exp(-\e^5pn^2/10^5)$.

Furthermore, using Proposition~\ref{Chern}, and as $|A||H_0|\geq \e^3n^2/30$, we have
\[
\P(\barr{F_2})\leq \P\Big(\sum_{a\in A}d_{G[V(H_0)]-H_0}(a)> (1+\e/8)p|A||H_0|\Big)\leq 4\exp (-\e^5pn^2/10^5).
\]

Finally, for each $U\subset B$ with $|U|\leq (1/2+\e/8)|H_0|$, we have $\E|\cup_{u\in U}(Y_u\cap E(G))|\leq (1/2+\e/8)p|A||H_0|$. Therefore, using Proposition~\ref{Chern}, as $|A||H_0|\geq \e^3n^2/30$, we have
\[
\P(\barr{F_3})\leq 2^{n}\cdot 4\exp(-\e^5pn^2/10^5)\leq \exp(-\e^5pn^2/(2\cdot 10^5)),
\]
for sufficiently large $C$ as $pn^2\geq Cn$.
Therefore, $\P(F_1\land F_2\land F_3)\geq 1-\exp(-\e^5pn^2/10^6)$. We will show that, if $F_1$, $F_2$, and $F_3$ hold, then $f(H_0,v)$ holds, completing the proof of the claim.

Suppose then that $F_1$, $F_2$, and $F_3$ hold, and let $H$ be a $(1/2+\e)$-residual subgraph of $G$ with $V(H)=V(H_0)$. For each $u\in B$, let $E_{u}=Y_u\cap E(H)$, so that, for each $e\in E_{u}$, $\{uv,e\}$ is a booster for $H_0$. Let $U$ be the set of vertices in $B$ for which $|E_{u}|\geq 50/\e\delta$. We need only show that $|U|\geq (1/2+\e/8)|H_0|$.

Now, taking $\d$ to be small, and then, depending on $\d$, $C$ to be large,
\begin{align*}
|\cup_{u\in U}(Y_u\cap E(G))|&\geq |\cup_{u\in B}(Y_u\cap E(G))|-|\cup_{u\in B\setminus U}(Y_u\cap E(G-H))|-\sum_{u\in B\setminus U}|E_u|
\\
&\geq |\cup_{a\in A}(Z_a\cap E(G))|-|\cup_{a\in A}(Z_a\cap E(G-H))|-(50/\e\delta)|B\setminus U|
\\
& \overset{F_1}{\geq} (1-\e/4)p|A||H_0|-\sum_{a\in A}(1/2-\e)d_{G[V(H_0)]}(a)-50|H_0|/\e\d
\\
&\overset{F_2}{\geq} (1-\e/4)p|A||H_0|-(1/2-\e)(1+\e/8)p|A||H_0|-2e(H_0)-10^4p|A||H_0|/\e^3\d C
\\
& \geq (1/2-\e/4+\e-\e/8-120\delta/\e^3- 10^4/\e^3\d C)p|A||H_0|
\\
& \geq (1/2+\e/4)p|A||H_0|.
\end{align*}
Therefore, as $F_3$ holds, we have $|U|\geq (1/2+\e/8)|H_0|$. This completes the proof of the claim.
\end{proof}
Thus, for sufficiently small $\delta$, and then sufficiently large $C$, the property in the lemma holds.
\end{proof}

Using Lemma~\ref{rotatemany}, we can now prove Lemma~\ref{rotateresil}, as follows.

\begin{proof}[Proof of Lemma~\ref{rotateresil}] Let $\d,C>0$ be constants for which Lemma~\ref{rotatemany} holds with $\e$.
With probability $1-o(n^{-3})$, $G$ has the property from Lemma~\ref{rotatemany} and, using Lemma~\ref{chernoff}, at most $pn^2$ edges. Suppose that $H_0$ is a 2-expander with at least $\e n$ vertices and at most $2\delta pn^2$ edges and that $H$ is a $(1/2+\e)$-residual subgraph of $G$ with $V(H)=V(H_0)$.
By the property from Lemma~\ref{rotatemany}, we can find a set $W$ of $|H|/8$ vertices $v\in V(H)$ for which there is some set $U_v\subset V(H)$ with $|U_v|\geq (1/2+\e/8)|H|$ and disjoint subsets $E_{uv}\subset E(H-H_0)$, $u\in U_v$, so that $|E_{uv}|\geq 50/\e\delta$ and, for each $u\in U_v$ and $e\in E_{uv}$, $\{uv,e\}$ is a booster for $H_0$.

Let $H_1$ be a random subgraph of $H-H_0$ with edges chosen independently at random with probability $\delta$. By Lemma~\ref{chernoff}, almost surely, $e(H_1)\leq 2\delta pn^2$.
For each $v\in W$, let $U'_v\subset U_v$ be the set of vertices $u$ for which $E_{uv}\cap E(H_1)\neq \emptyset$, so that $\P(u\notin U'_v)\leq (1-\delta)^{50/\e\delta}\leq \exp(-50/\e)\leq \e/32$. Therefore, as the sets $E_{uv}$, $u\in U_v$, are disjoint, we have, by Lemma~\ref{chernoff},
\[
\P(|U'_v|\geq (1/2+\e/16)|H_0|)\geq \P(|U'_v|\geq (1-\e/32)|U_v|)=1-o(n^{-1}).
\]
Therefore, some such graph $H_1\subset H$ with $e(H_1)\leq 2\delta pn^2$ exists with $|U'_v|\geq (1/2+\e/16)|H_0|$ for each $v\in W$.
\end{proof}

\subsection{Finding many boosters in $H$}\label{secextend}
Lemma~\ref{rotateresil} demonstrates that there are likely to be many boosters for a sparse expander $H_0\subset G$ with help from some other sparse subgraph $H_1$. We now show that many of these boosters are likely to exist in $G$, sufficiently many that some exist in any $(1/2+\e)$-residual subgraph $H$ of $G$.

\lem\label{extend} Let $0<\e\leq 1$. There exists $\d,C>0$ such that the random graph $G=\GG(n,p)$, with $p\geq C/n$, has the following property with probability $1-o(n^{-3})$. Suppose~$H_0$ and $H_1$ are edge-disjoint subgraphs of $G$ with $V(H_0)=V(H_1)$, $e(H_0),e(H_1)\leq 2\d pn^2$ and $|H_0|\geq \e n$, and where $H_0$ has $\e$-many boosters with help from $H_1$. Then, for any $1/2$-residual subgraph $H\subset G$ with $V(H)=V(H_0)$ there is some $e_1,e_2\in E(H)\cup E(H_1)$ such that $\{e_1,e_2\}$ is a booster for $H_0$.
\ma
\pr Let $\d,C>0$ be determined later, where $\delta$ will be taken to be small depending on $\e$ and $C$ to be large depending on $\delta$. Let $\mathcal{H}$ be the set of pairs $(H_0,H_1)$ where $H_0$ and $H_1$ are edge-disjoint subgraphs with $V(H_0)=V(H_1)\subset [n]$, $e(H_0),e(H_1)\leq 2\d pn^2$, and $|H_0|\geq \e n$, and where $H_0$ has $\e$-many boosters with help from $H_1$.

Fix $(H_0,H_1)\in \mathcal{H}$, and, for each $x\in V(H_0)$, let $V_x$ be the set of vertices $v\in V(H_0)\setminus \{x\}$ for which there is some edge $e\in H_1$ so that $\{xv,e\}$ is a booster for $H_0$. Let $X=\{x\in V(H_0):|V_x|\geq (1/2+\e)|H_0|\}$, so that, by Definition~\ref{emany}, $|X|\geq \e|H_0|\geq \e^2 n$. Let $p\geq C/n$ and $G=G(n,p)$, and let $f(H_0,H_1)$ be the event that $\sum_{x\in X}d_{G-H_0-H_1}(x,V_x)\geq (1/2+\e/4)p |X||H_0|$. Note that
\begin{align*}
\E\Big(\sum_{x\in X}d_{G-H_0-H_1}(x,V_x)\Big) &\geq (1/2+\e)p |X||H_0|-2e(H_0)-2e(H_1)
\\
&\geq (1/2+\e/2)p |X||H_0|\geq \e^3pn^2/2.
\end{align*}
Therefore, for $C$ sufficiently large, by Proposition~\ref{Chern}, we have
\begin{equation}\label{thisone}
\P(\barr{f(H_0,H_1)})\leq 4\exp(-(\e/4)^2\cdot \e^3pn^2/18)\leq \exp(-\e^5pn^2/10^3).
\end{equation}
Note that the events $f(H_0,H_1)$, $H_0\subset G$ and $H_1\subset G$ are independent. Therefore, the probability that $f(H_0,H_1)$ does not hold and $H_0,H_1\subset G$ for some $(H_0,H_1)\in\mathcal{H}$ is at most
\begin{align*}
\sum_{(H_0,H_1)\in\mathcal{H}}\P(\barr{f(H_0,H_1)} \wedge H_0,H_1\subset G) &= \sum_{(H_0,H_1)\in\mathcal{H}}\P(\barr{f(H_0,H_1)})\P(H_0\subset G)\P(H_1\subset G ) \\
&\overset{\eqref{thisone}}{\leq}  \exp(-\e^5pn^2/10^3)\cdot \left(\sum_{H_0\subset K_n,e(H_0)\leq 2\d pn^2}\P(H_0\subset G)\right)^2\\
&\leq \exp(-(\e^5/10^3-4\d\log(e/2\delta))pn^2),
\end{align*}
using Proposition~\ref{expectedH}.
Therefore, choosing $\d$ sufficiently small, and $C$ sufficiently large, as $p\ge C/n$, with probability $1-o(n^{-3})$, $f(H_0,H_1)$ holds for every $(H_0,H_1)\in \mathcal{H}$ with $H_0,H_1\subset G$.

By this, and Lemma \ref{AKS1equiv}, with probability $1-o(n^{-3})$ we can assume that $f(H_0,H_1)$ holds for every $(H_0,H_1)\in \mathcal{H}$ with $H_0,H_1\subset G$, and, if $X,U\subset V(G)$, $|X|\geq \e^2n$, and $|U|\geq \e n$, then $e_G(X,U)\leq (1+\e/6)p|X||U|$.

Now, take any $(H_0,H_1)\in \mathcal{H}$ with $H_0,H_1\subset G$. Using the notation for a fixed $(H_0,H_1)\in \mathcal{H}$ above, as $f(H_0,H_1)$ holds,
\[
\sum_{x\in X}(d_{G-H_0-H_1}(x,V_x)-d_{G[V(H_0)]}(x)/2)\geq (1/2+\e/4)p|X||H_0| -e_G(X,V(H_0))/2> 0.
\]
Thus, there is some $x\in V(H_0)$ with $d_G(x,V_x)>d_{G[H_0]}(x)/2$. For any $1/2$-residual subgraph $H\subset G$ with $V(H)=V(H_0)$, then, we have $d_H(x,V_x)\geq d_G(x,V_x)-d_{G[V(H_0)]}(x)/2> 0$. That is, there must be some $v\in N_{H}(x)\cap V_x$. By the definition of $V_x$ at the start of the proof, there is some $e\in E(H_1)$ such that $\{e,xv\}$ is a booster for $H_0$. As this holds for each $(H_0,H_1)\in \mathcal{H}$ with $H_0,H_1\subset G$ and each $1/2$-residual subgraph $H\subset G$ with $V(H)=V(H_0)$, this completes the proof.
\oof

\subsection{Proof of Theorem~\ref{newtheorem}}\label{secnewthmproof}

Armed with Lemma~\ref{rotateresil} and Lemma~\ref{extend}, we can now prove Theorem~\ref{newtheorem}.

\begin{proof}[Proof of Theorem~\ref{newtheorem}] Let $\d, C>0$ be such that Lemma~\ref{rotateresil} holds for $\e$, Lemma~\ref{extend} holds for $\e/16$, and, furthermore, let $C$ be sufficiently large that $C\delta\geq 2$. Letting $p\geq C/n$, then, with probability $1-o(n^{-3})$, $G=G(n,p)$ has the property in Lemma~\ref{rotateresil} with $\e$ and the property in Lemma~\ref{extend} with $\e/16$. Let $H$ be a $(1/2+\e)$-residual subgraph of $G$ with $|H|\geq \e n$ which contains a spanning $2$-expander, $H_0$ say, with at most $\delta pn^2$ edges.

For each $1\leq i\leq n$, find $e_{i,1},e_{i,2}\in E(H)$ such that $\{e_{i,1},e_{i,2}\}$ is a booster for $H_{i-1}$, and let $H_i=H_{i-1}+e_{i,1}+e_{i,2}$. This is possible, for each $1\leq i\leq n$, as follows. Noting that $e(H_{i-1})\leq \delta pn^2+2n\leq 2\delta pn^2$, by the property from Lemma~\ref{rotateresil}, there is some subgraph $H'\subset H-H_{i-1}$ with $e(H')\leq 2\delta pn^2$ so that $H_{i-1}$ has $(\e/16)$-many boosters with help from~$H'$. Therefore, from the property from Lemma~\ref{extend} there is some $e_{i,1},e_{i,2}\in E(H)$ such that $\{e_{i,1},e_{i,2}\}$ is a booster for $H_{i-1}$, as required.

We have added $n$ boosters to $H_0$ to get $H_n\subset H$, and therefore $H_n$, and hence $H$, is Hamiltonian.
\end{proof}


\section{Proof of Lemma~\ref{nonresil}}\label{secnonres}
We will prove Lemma~\ref{nonresil}, for each $\e>0$, by showing that, in a typical linear vertex subset $U$ of a binomial random graph $G=G(n,p)$, $G[U]$ can be made into an unbalanced bipartite graph without deleting any more than $(1/2+\e)$ of the edges around any vertex.
\begin{proof}[Proof of Lemma~\ref{nonresil}] Note that we may assume $0<\e<1/8$, and let $C=10^3/\e^7$.
By Lemma~\ref{AKS1equiv}, if $p\geq C/n$, then with probability $1-o(n^{-3})$, $G=G(n,p)$ has the following property.

\begin{enumerate}[label = \textbf{P}]
\item If $A,B\subset V(G)$ and $|A||B|\geq \e^5n^2$, then $(1-\e/3)p|A||B|\leq  e_G(A,B)\leq (1+\e/3)p|A||B|$. \label{Rnew}
\end{enumerate}

Now, let $U\subset V(G)$ satisfy $|U|\geq \e n$. We will show that $G[U]$ is not $(1/2+\e)$-resiliently Hamiltonian.
Pick a vertex partition $A\cup B=U$, so that $e_{G}(A,B)$ is maximised. For each $v\in A$, $d_{G}(v,A)\leq d_{G}(v,B)$, otherwise, moving $v$ from $A$ into $B$ would increase $e_{G}(A,B)$. Similarly, for each $v\in B$, $d_{G}(v,B)\leq d_{G}(v,A)$. Therefore, if $H=G[A]\cup G[B]$, then, for each $v\in V(G)$, $d_H(v)\leq d_{G[U]}(v)/2$. If $|A|\neq |B|$, then, as $G[U]-H$ is an unbalanced bipartite graph, $G[U]-H$ is not Hamiltonian, and thus, if $|A|\neq |B|$, $G[U]$ is not $(1/2)$-resiliently Hamiltonian.

Suppose then that $|A|=|B|$. Let $U_0\subset U$ be the set of vertices with degree at most $1/\e$ in $G[U]$.
\begin{claim} $|U_0|<\e^3 n$.\label{mornclaim}
\end{claim}
\begin{proof}[Proof of Claim~\ref{mornclaim}] Suppose, for contradiction, that we may take a set $U'\subset U_0$ with $|U'|=\e^3 n$. There are at most $|U'|/\e=\e^2n$ edges between $|U'|$ and $|U\setminus U'|$. However, as $|U'||U\setminus U'|\geq \e^3 n\cdot |U|/2\geq \e^5 n^2$, by~\ref{Rnew}, there are at least $\e^5pn^2/2> \e^2n$ edges between $|U'|$ and $|U\setminus U'|$, a contradiction.
\end{proof}

From Claim~\ref{mornclaim}, we have that $|U_0\cup N_{G[U]}(U_0)|\leq 2\e^2 n$. Let $A_0=A\setminus (U_0\cup N_{G[U]}(U_0))$.  As $|A|=|B|$, we have $|A|,|B|\geq \e n/2$, and hence $|A_0|\geq \e n/4$.
Using~\ref{Rnew}, and as $|A|=|B|$,
\begin{align*}
\sum_{x\in A_0}d(x,B) - \left(1+\e\right)\sum_{x\in A_0}d(x,A)
&=e(A_0,B)-\left(1+\e\right)e(A_0,A)
\\
&\leq \left(1+\e/3\right)p|A_0||B|-\left(1+\e\right)\left(1-\e/3\right)p|A_0||A|\leq 0.
\end{align*}
Therefore, there exists $x\in A_0$ with $d(x,B)\leq (1+\e)d(x,A)$, so that $d(x,B)\leq (1+\e)d_{G[U]}(x)/2$. Let $A'=A\setminus\{x\}$ and $B'=B\cup\{x\}$ be a new partition of $U$. Take $H'=G[A']\cup G[B']$, so that, by the choice of $x$, $d_{H'}(x)\leq (1+\e)d_{G[U]}(x)/2$. For each vertex $v\in U\setminus N_{G[U]}(x)$, with $v\neq x$, we did not change where its neighbours lie in the partition, so $d_{H'}(v)\leq d_{G[U]}(v)/2$. For each vertex $v\in N_{G[U]}(x)$, we have $v\notin U_0$ by the choice of $A_0$, so that $d_{G[U]}(v)\geq 1/\e$, and we moved one of its neighbours, $x$, across the partition. Thus, $d_{H'}(v)\leq d_{G[U]}(v)/2+1\leq (1/2+\e)d_{G[U]}(v)$. As $|A'|\neq |B'|$, $G[U]-H'$ is an unbalanced bipartite graph, and hence not Hamiltonian. Thus, $G[U]$ is not $(1/2+\e)$-resilently Hamiltonian.
\end{proof}


\section{Resilient large set expansion}\label{secresil}

Set against small set expansion, it is comparatively straightforward to show that our random graphs resiliently contain a sparse subgraph in which large sets expand. We show this for both Theorem~\ref{hamres} and Theorem~\ref{coreres}, through the following definition and theorem.

\de A graph $H$ is a \emph{$(2,k)$-expander} if, for every subset $U\subset V(H)$, with $|U|\leq k$, we have $|N(U)|\geq 2|U|$.
\fn

Note that, in contrast to the definition of a 2-expander (Definition~\ref{expdefn}), a $(2,k)$-expander need not be connected.

\begin{theorem}\label{newnewtheorem}
For each $\e>0$, there exists $\d,C>0$ such that, if $p\geq C/n$, then $G=G(n,p)$ has the following property with probability $1-o(n^{-3})$. Any $(1/2+\e)$-residual subgraph $H$ of $G$, with $|H|\geq \e n$, which contains a spanning $(2,C/p)$-expander with at most $\delta pn^2$ edges is Hamiltonian.
\end{theorem}
\begin{proof}
By Theorem~\ref{newtheorem}, we may take $\delta>0$ and $C\geq 200/\e^3\delta$ such that, if $p\geq C/n$ then we have the following property with probability $1-o(n^{-3})$.

\begin{enumerate}[label=\bfseries Q\arabic*]\addtocounter{enumi}{0}
\item Any $(1/2+\e)$-residual subgraph $H\subset G$, with $|H|\geq \e n$, which contains a spanning $2$-expander with at most $2\delta pn^2$ edges is Hamiltonian.\label{RR2n}
\end{enumerate}

Let $p\geq C/n$ and $G=G(n,p)$, so that with probability $1-o(n^{-3})$ we have both~\ref{RR2n} and, by Lemma~\ref{AKS1equiv}, the following property.

\begin{enumerate}[label=\bfseries Q\arabic*]\addtocounter{enumi}{1}
\item  If $A, B\subset V(G)$ satisfy $|A|\geq C/p$ and $|B|\geq \e n/2$, then \label{RR1n}
\[
(1-\e/3)p|A||B|\leq e(A,B)\leq (1+\e/3)p|A||B|.
\]
\end{enumerate}

We will show that, for sufficiently large $n$, $G$ satisfies the property in the theorem.
Let $H$ be a $(1/2+\e)$-residual subgraph of~$G$, with $|H|\geq \e n$, which contains a spanning $(2,C/p)$-expander $H_1$ with $e(H_1)\leq \delta pn^2$. 
By~\ref{RR1n}, for all subsets $A\subset V(H)$, $B\subset V(H)\setminus A$ with $|A|\geq C/p$ and $|B|\geq |H|/2$, we have
\begin{align}
e_H(A,B)&\geq e_G(A,B)-(1/2-\e)e_G(A,V(H)) \geq (1-\e/3)p|A||B|-(1/2-\e)(1+\e/3)|A||H|
\nonumber \\
&\geq (1-\e/3)p|A||B|-(1-\e)p|A||B|= 2\e p|A||B|/3\geq C\e^2 n/3.\label{hmm}
\end{align}

Let $H_2\subset H$ be a random subgraph of $H$ where each edge is included with probability $\delta/2$. Therefore, the probability there exists some disjoint $A,B\subset V(H)$ for which $|A|\geq C/p$, $|B|\geq |H|/2$ and $e_{H_2}(A,B)=0$ is, by~\eqref{hmm}, at most
\[
2^n\cdot 2^n\cdot (1-\delta/2)^{C\e^2 n/3}\leq 4^ne^{-C \e^2\delta n/6}=o(1).
\]
By~\ref{RR1n}, $e(G)\leq pn^2$, and thus, by Lemma~\ref{chernoff}, $H_2$ almost surely has at most $\delta pn^2$ edges.  Therefore, we may take some $H_2\subset H$ such that $e(H_2)\leq \delta pn^2$ and, for each disjoint $A,B\subset V(H)$ with $|A|\geq C/p$ and $|B|\geq |H|/2$, we have $e_{H_2}(A,B)>0$. Let $H_0=H_1\cup H_2\subset H$, noting that $e(H_0)\leq 2\delta pn^2$. We will show that $H_0$ is a 2-expander, so that, by~\ref{RR2n}, $H$ is Hamiltonian, as required.

Firstly, for each $A\subset V(H)$ with $C/p\leq |A|\leq |H|/8$, there are no edges between $A$ and $V(H)\setminus (A\cup N_{H_2}(A))$ in $H_2$, and, therefore, $|A\cup N_{H_2}(A)|\geq|H|/2$. Thus, $|N_{H_2}(A)|\geq |H|/2-|A|\geq 3|H|/8\geq 2|A|$.

Secondly, the graph $H_0$ is connected. Indeed, suppose to the contrary that $H_0$ has a component with vertex set $A$ for which $|A|\leq |H|/2$. Then, as $H_1$ is a $(2,C/p)$-expander, from $N_{H_0}(A)=\emptyset$ we have $|A|>C/p$. As there are no edges between $A$ and $V(H)\setminus A$ in $H_2$, this gives a contradiction.

As $H_1$, and thus $H_0$, is a $(2,C/p)$-expander, $H_0$ is a $2$-expander, as required.
\end{proof}


\section{Born resilience of Hamiltonicity}\label{secfull}

By Lemma~\ref{switch} (with $p=M/\binom{n}{2}$) and Theorem~\ref{newnewtheorem}, to prove Theorem~\ref{hamres} it is sufficient to show the following lemma.

\begin{lemma}\label{smallexpall} For each $\delta,C >0$, in almost every random graph process $\{G_M\}_{\geq 0}$ with $n$ vertices, for each $M$ and $\e>0$, if $\delta(G_M)\geq 2$, then any spanning $(1/2+\e)$-residual subgraph $H\subset G_M$ contains a spanning $(2,C\binom{n}{2}/M)$-expander with at most $\delta M$ edges.
\end{lemma}

Typically, in graphs in $\{G_M\}_{M\geq 0}$, sets of vertices with small degree (here, $\leq M/10^3n$) will resiliently expand as these vertices form an independent set (see Section~\ref{secsmall}). Typically, sets with vertices with larger degree must resiliently expand to prevent overly dense subgraphs (see Section~\ref{secmed}). Combining this, we prove Lemma~\ref{smallexpall} in Section~\ref{secallsize}.

\subsection{Small degree vertices in the random graph process}\label{secsmall}
As is well-known, in almost every random graph process $\{G_M\}_{M\geq 0}$ with $n$ vertices, the first graph with minimum degree at least 2 has around $n(\log n+\log\log n)/2$ edges, as follows (see, for example, Bollob\'as~\cite[Theorems~2.2(ii) and 3.5]{bollorand}).

\begin{lemma}\label{isol} If $M= n(\log n+\log\log n-\omega(1))/2$, then, almost surely, $\delta(G_{n,M})\leq 1$. If $k\in \N$ is fixed, and $M=n(\log n+(k-1)\log\log n+\omega(1))/2$, then, almost surely, $\delta(G_{n,M})\geq k$.
\end{lemma}

Furthermore, in almost every $n$-vertex random graph process $\{G_M\}_{M\geq 0}$, if $M\geq 25n\log n$, then there are no vertices with small degree in $G_M$.

\begin{lemma}\label{mind} In almost every random graph process $\{G_M\}_{M\geq 0}$ with $n$ vertices, if $M\geq 25n\log n$, then $\delta(G_M)\geq M/n$.
\end{lemma}
\begin{proof} For each $M\geq 25\log n$, let $p_M=M/\binom{n}{2}$ and $G=G(n,p)$. For each $v\in V(G)$,  $\E(d_{G}(v))=(n-1)p_M=2M/n\geq 50\log n$, so that, by Lemma~\ref{chernoff},
\[
\P(d_{G}(v)\leq M/n)\leq 2\exp(-50\log n/12)=o(n^{-4}).
\]
Therefore, by Lemma~\ref{switch}, in almost every $n$-vertex random graph process $\{G_M\}_{M\geq 0}$, in every $G_M$ with $M\geq 25n\log n$, we have $\delta(G_M)\geq M/n$.
\end{proof}

As shown by Bollob\'as~\cite{bollo84}, in almost every random graph process, when the graph first has minimum degree 2 there are no vertices with low degree which are close together (see Lemma~\ref{bollopath}).  We wish to show that, in almost every random graph process, such a property holds for \emph{every} graph with minimum degree at least 2. Therefore, we repeat the argument in~\cite{bollo84} to record some more detail.

\lem\label{bollopath} Fix an integer $k>1$, and let $n,M\in \N$. If $19n\log n/40\leq M\leq 27n\log n$, then, with probability $1-o(n^{-1/4})$, no two vertices with degree at most $\log n/ 36$ are within distance $k$ of each other in the random graph $G=\GG_{n,M}$.
\ma
\pr
Let $N=\binom{n}{2}$ and $d=\log n/36$. For each $a,b\in \N$ with $a\leq b$, use $(b)_a$ for the falling factorial $b(b-1)\cdots (b-a+1)$, and recall the simple result that, if $0\leq a\leq b\leq c$, then $(b)_a/(c)_a\leq (b/c)^a$. Note also that, for sufficiently large $n$,
\begin{equation}\label{soreeyes}
\frac{Mn}{N-2n}=\frac{2M}{n-5}\leq 60\log n\;\;\;\text{ and }\;\;\;
(2n-3d)M/N\geq 1.8\log n.
\end{equation} 

The expected number of paths in $G=\GG_{n,M}$ with some length $1\leq i\leq k$ whose endpoints together have degree at most some $j\leq 2d$ is at most
\begin{align*}
\sum_{i=1}^kn^{i+1}\sum_{j=1}^{2d}\binom{2n}{j}\binom{N-2n+3}{M-j-i}\Big/\binom{N}{M}
&=\sum_{i=1}^kn^{i+1}\sum_{j=1}^{2d}\binom{2n}{j}\cdot \frac{(N-M)_{2n-3-i-j}}{(N)_{2n-3-i-j}}\cdot \frac{(M)_{j+i}}{(N-2n+3+j+i)_{j+i}}
\\
&\leq\sum_{i=1}^kn^{i+1}\sum_{j=1}^{2d}\binom{2n}{j}\cdot \left(1-\frac{M}{N}\right)^{2n-3-j-i}\cdot \left(\frac{M}{N-2n}\right)^{j+i}
\\
&\leq 2dk\cdot n^{k+1}\cdot \left(\frac{en}{d}\right)^{2d}\cdot e^{-(2n-3d)M/N}
\cdot \left(\frac{M}{N-2n}\right)^{2d+k}\\
&\leq 2dkn\cdot \left(\frac{Mn}{N-2n}\right)^k\cdot \left(\frac{eMn}{d(N-2n)}\right)^{2d}\cdot e^{-(2n-3d)M/N}\\
&\overset{\eqref{soreeyes}}{\leq} 2dkn\cdot (60\log n)^k\cdot \left(\frac{60e\log n}{d}\right)^{2d}\cdot e^{-1.8\log n}\\
&\leq 2dk\cdot (60\log n)^k\cdot e^{18d}\cdot n^{-0.8}\\
&\leq 2dk\cdot (60\log n)^k\cdot e^{\log n/2}\cdot n^{-0.8}=o(n^{-1/4}).
\end{align*}
Thus, with probability $1-o(n^{-1/4})$, no two vertices with degree at most $d$ are within distance $k$ of each other in $G$.
\oof

The property in Lemma \ref{bollopath} is non-monotone, so we need slightly more work to show that it holds throughout almost every random graph process.

\lem\label{mepath} Let $k>1$ be fixed.
In almost every random graph process $\{G_M\}_{M\geq 0}$ with $n$ vertices, if $19n\log n/40\leq M\leq 25n\log n$, then there are no two vertices of degree at most $\log n/40$ within a distance $k$ of each other in $G_M$.
\ma
\pr Let $N=\binom{n}{2}$, $d=\log n/36$ and $M_0=n^{7/8}$. Letting $p=M_0/N\leq 3n^{-9/8}$, the probability the random graph $\GG(n,p)$ has a vertex with degree at least $d/10$ is at most
\[
n\binom{n}{d/10}p^{d/10}\leq n\left(\frac{10e n p}{d}\right)^{d/10}\leq n\left(\frac{30e}{dn^{1/8}}\right)^{d/10}=o(n^{1-d/100})=o(n^{-3}).
\]
From this, and Lemma~\ref{switch}, we have that, in almost every graph process $\{G_M\}_{M\geq 0}$ with $n$ vertices, for each $j$, $0\leq j\leq N-M_0$, $\Delta(G_{j+M_0}-G_j)\leq d/10$.

There are at most $25n^{1/8}\log n$ values of $j\in\N$ for which $19n\log n/40\leq jM_0\leq 25n\log n+M_0$. By Lemma \ref{bollopath} and a union bound, almost surely, for each such $j$, $G_{jM_0}$ contains no two vertices with degree at most $d$ within a distance $k$ of each other. Now, suppose for some $M$, $19n\log n/40\leq M\leq 25n\log n$, $G_M$ has two vertices $x$ and $y$ with degree at most $\log n/40$ which are at most distance $k$ apart in $G_M$. Find the smallest $j$ for which $jM_0\geq M$. Then $19n\log n/40\leq jM_0\leq 25n\log n+M_0$, but in $G_{jM_0}$ both $x$ and $y$ have degree at most $\log n/40+d/10=d$ and are still within a distance $k$ of each other in $G_{jM_0}$, a contradiction. Therefore, the property in the lemma almost surely holds for all $M$ with $19n\log n/40\leq M\leq 25n\log n$.
\oof

\begin{corollary}\label{farmin} In almost every random graph process $\{G_M\}_{M\geq 0}$ with $n$ vertices, in each $G_M$ with $M\geq 19n\log n/40$ there are no two vertices with degree at most $M/10^3n$ within a distance 5 of each other.
\end{corollary}
\begin{proof}
Almost surely, by Lemma \ref{mepath}, in each $G_M$ with $19n\log n/40\leq M\leq 25n\log n$ there are no two vertices with degree less than $\log n/40\geq M/10^3n$ within a distance $5$ of each other. Almost surely, by Lemma \ref{mind}, for each $G_M$ with $M\geq 25n\log n$, we have $\delta(G_M)> M/10^3n$, so no vertices with degree at most $M/10^3n$ exist.
\end{proof}

\subsection{Expansion from minimum degree conditions}\label{secmed}
Minimum degree conditions occur naturally in our $(1/2+\e)$-residual subgraphs in Lemma~\ref{smallexpall}. The following lemma will allow us to convert these conditions into expansion properties.

\lem\label{mindegexp3} For each $\d>0$, there exists $C>0$ such that the following holds for $p\geq C/n$. With probability $1-o(n^{-3})$, in the random graph $G=\GG(n,p)$ there are no two sets $A,B\subset V(G)$ with $|A|\leq 1/\delta p$, $|B|\leq 10|A|$ and $e(A,B)\geq \d pn|A|$.
\ma
\pr 
For each value of $t$, with $1\leq t\leq 1/\delta p$, let $p_t$ be the probability that there are two sets $A,B\subset V(G)$ with $|A|=t$, $|B|=10t$ and $e(A,B)\geq \d pnt$. Note that for such a pair $A$ and $B$, considering that some edges may be counted twice in $e(A,B)$ (if $A$ and $B$ overlap), then there are at least $\d pnt/2$ distinct edges appearing between $A$ and $B$. Thus, for sufficiently large $C$,
\begin{align*}
p_t\leq \binom{n}{t}\binom{n}{10t}\binom{10t^2}{\d pnt/2}p^{\d pnt/2}
&\leq \left(\frac{en}{t}\right)^{11t}\left(\frac{20et}{\d n}\right)^{\d pnt/2} \leq \left(\frac{20e^2}{\d}\right)^{11t}\left(\frac{20et}{\d n}\right)^{\d Ct/2-11t}
\\
&\leq \left(\frac{400e^3t}{\d^2 n}\right)^{11t}.
\end{align*}
Therefore, if $1\leq t<\log n$, then,  $p_t=o((1/\sqrt{n})^{11})=o(n^{-4})$. If $\log n\leq t\leq 1/\delta p$, then $p_t\leq (400e^3/\delta^3 pn)^{11\log n}\leq (400e^3/\delta^3 C)^{11\log n}=o(n^{-4})$, for sufficiently large $C$. Therefore, no two such sets exist for any $t\leq 1/\delta p$ with probability $1-o(n^{-3})$.
\oof

\subsection{Resilient small set expansion}\label{secallsize}
We can now combine the work in Sections~\ref{secsmall} and~\ref{secmed} to prove Lemma~\ref{smallexpall}.
\begin{proof}[Proof of Lemma~\ref{smallexpall}]
Note that we can assume that $\d\leq 10^{-4}$ and $C\geq 16/\d$.
Almost surely, by Lemma~\ref{isol}, for each $G_M$ with $M\leq 19n\log n/40$ we have $\delta(G_M)<2$, so we need only consider the range $M\geq 19n\log n/40$.
By Corollary~\ref{farmin}, for each $M\geq 19n\log n/40$, there are no vertices with degree at most $M/10^3n$ within distance $5$ of each other.

By Lemma \ref{switch}, and Lemma \ref{mindegexp3} applied with $\delta'=\min\{\delta/8,1/C\}$, we almost surely get the following property for each graph $G_M$, $M\geq 19n\log n/40$, with $p=M/\binom{n}{2}\geq 19\log n/20$.
\begin{enumerate}[label=\bfseries R]\addtocounter{enumi}{0}
\item There are no two sets $A,B\subset V(G_M)$ with $|A|\leq C/p$, $|B|\leq 10|A|$ and $e(A,B)\geq \d pn|A|/8$.
\label{RR2}
\end{enumerate}
The random graph process then has the property in the lemma, as follows.

Fixing $M\geq 19n\log n/40$, $p=M/\binom{n}{2}$ and $\e>0$, suppose $\delta(G_M)\geq 2$ and let $H$ be a $(1/2+\e)$-residual spanning subgraph of~$G_M$. Note that $\delta(H)\geq 2$. For each $v\in V(H)$, take $\min\{d_H(v),\delta pn/4\}$ edges incident to $v$, and use these edges to create a new graph $H_0$ with $e(H_0)\leq \delta pn^2/4\leq \delta M$ and $d_{H_0}(v)\geq \min\{d_H(v),\delta pn/4\}\geq 2$ for each $v\in V(H)$.

Let $A\subset V(H)$ with $|A|\leq C/p$. Let $A_1\subset A$ be the set of vertices not within distance 2 of another vertex in $A$ in $H_0$, and let $A_2=A\setminus A_1$. As $\delta(H_0)\geq 2$, $|N_{H_0}(A)|\geq 2|A_1|+|N_{H_0}(A_2)|$. Now, two vertices in $A$ with degree at most $M/10^3n$ in $G_M$ cannot be within distance 2 of the same vertex in $A$ in $H_0$ or each other. Therefore, at least $|A_2|/2$ vertices in $A_2$ must have degree at least $M/10^3n$ in $G_M$, and hence degree $\delta pn/4$ in $H_0$. Thus, $e(A_2,A_2\cup N_{H_0}(A_2))\geq \delta pn|A_2|/8$, so that, by~\ref{RR2}, $|N_{H_0}(A_2)|\geq 2|A_2|$. Therefore, $|N_{H_0}(A)|\geq 2|A_1|+|N_{H_0}(A_2)|\geq 2|A|$. Hence, as required, $H_0\subset H$ is a $(2,C/p)$-expander with $e(H_0)\leq \delta M$ which spans $H$.
\end{proof}

\newpage

\section{Born resilience of Hamiltonicity in the $k$-core}\label{seccore}

Our key result in this section is the following lemma, showing that the $k$-core of $G(n,p)$ almost always resiliently contains a sparse spanning $(2,\Theta(1/p))$-expander.

\begin{lemma}\label{coresparseexpand} For each $\d,C>0$, there exists $k_0$ such that, for each $k\geq k_0$, $n\in \N$ and $p>0$ with $k/n\leq p\leq 2\log n/n$, the following holds. With probability $1-o(n^{-3})$, $G=G(n,p)$ has the property that any spanning $1/2$-residual subgraph of $G^{(k)}$ contains a spanning $(2,C/p)$-expander with at most $\delta pn^2$ edges.
\end{lemma}

Following Krivelevich, Lubetzky and Sudakov~\cite{KLS14}, we will consider two cases: the \emph{critical case} where $p\leq 4k/n$ (in Section~\ref{seccrit}) and the \emph{supercritical} case where $p\geq 4k/n$ (in Section~\ref{secsuper}). We use a lemma from~\cite{KLS14}, but note that its proof relies on a relatively simple calculation and that our methods for demonstrating Hamiltonicity in the $k$-core (for large $k$) are quite different from those used by Krivelevich, Lubetzky and Sudakov in~\cite{KLS14}.


We will show first, in Section~\ref{propsec}, that Theorem~\ref{coreres} follows from Lemma~\ref{coresparseexpand}, Theorem~\ref{newnewtheorem} and some standard results on the likely size and appearance of the $k$-core.

\subsection{Typical properties of the $k$-core}\label{propsec}
\L uczak~\cite{Luc87} showed that, when the $k$-core first appears in the random graph process it is likely to be linear in size, as follows.

\begin{theorem}[\cite{Luc87}]\label{corelinear} For each $k\geq 3$, in almost every $n$-vertex random graph process $\{G_M\}_{M\geq 0}$, for each $M$, either $G^{(k)}=\emptyset$ or $|G^{(k)}|\geq n/5000$.\hfill\qed
\end{theorem}

The threshold of appearance for the $k$-core has been well-studied (see~\cite{Luc91,PSW96,PVW11}), showing that, as the first $kn/2$ edges are added in the  $n$-vertex random graph process, the $k$-core is likely to be empty, as follows (see also~\cite{KLS14}).
\begin{theorem}\label{coreempty} For each $k\geq 20$, in almost every  $n$-vertex random graph process $\{G_M\}_{M\geq 0}$, if $M\leq kn/2$, then $G_M^{(k)}=\emptyset$.
\end{theorem}

Theorem~\ref{coreres} is implied by Lemma~\ref{coresparseexpand} and Theorems~\ref{newnewtheorem},~\ref{corelinear}, and~\ref{coreempty}, as follows.

\begin{proof}[Proof of Theorem~\ref{coreres}] 
Using Theorem~\ref{newnewtheorem}, let $\d,C>0$ be such that, if $p\geq C/n$, then $G=G(n,p)$ has the following property with probability $1-o(n^{-3})$.
\begin{enumerate}[label = \textbf{S\arabic*}]\addtocounter{enumi}{0}
\item If $H$ is a $(1/2+\e)$-residual subgraph of $G$ with $|H|\geq n/5000$ which contains a spanning $(2,C/p)$-expander with at most $\delta pn^2$ edges, then $H$ is Hamiltonian. \label{R41}
\end{enumerate}
Using Lemma~\ref{coresparseexpand}, let $k_0\geq C$ be such that, for each $k\geq k_0$ and $k/n\leq p\leq 2\log n/n$, $G=G(n,p)$ has the following property with probability $1-o(n^{-3})$.
\begin{enumerate}[label = \textbf{S\arabic*}]\addtocounter{enumi}{1}
\item If $G^{(k)}\neq\emptyset$, then any spanning $1/2$-residual subgraph of $G^{(k)}$ contains a spanning $(2,C/p)$-expander with at most $\delta pn^2$ edges. \label{R42}
\end{enumerate}
Increase $k_0$ if necessary, so that, using Lemmas~\ref{switch} and~\ref{nonresil}, we have the following property for almost every $n$-vertex random graph process $\{G_M\}_{M\geq 0}$.
\begin{enumerate}[label = \textbf{S\arabic*}]\addtocounter{enumi}{2}
\item For each $M\geq k_0n/2$, there is no $U\subset V(G_M)$ with $|U|\geq n/5000$ which is $(1/2+\e)$-resiliently Hamiltonian. \label{R43}
\end{enumerate}

Now, let $k\geq k_0$, and consider the random graph process $\{G_M\}_{M\geq 0}$ with $n$ vertices. By Lemma~\ref{isol}, $\delta(G_{M})$ is almost surely at least $k$ for each $M\geq (n-1)\log n$, in which case $G_M^{(k)}=G_M$. Thus, almost surely, by Theorem~\ref{hamres}, for each $M\geq (n-1)\log n$,  $G_M$ is $(1/2-\e)$-resiliently Hamiltonian but not $(1/2+\e)$-resiliently Hamiltonian. Furthermore, by Lemma~\ref{coreempty}, almost surely, if $M\leq kn/2$, then $G^{(k)}_M=\emptyset$. Therefore, we can assume that $kn/2\leq M\leq (n-1)\log n$.

Almost surely,~\ref{R43} holds and, for each $kn/2\leq M\leq (n-1)\log n$, $G_M$ satisfies~\ref{R41} and~\ref{R42} with $p=M/\binom{n}{2}$, so that $k\leq pn\leq 2\log n$. Almost surely, by Theorem~\ref{corelinear}, for each $kn/2\leq M\leq (n-1)\log n$, $G_M^{(k)}=\emptyset$ or $|G_M^{(k)}|\geq n/5000$. Then, if $G_M^{(k)}\neq \emptyset$, by~\ref{R41} and~\ref{R42}, $G_M^{(k)}$ is $(1/2-\e)$-resiliently Hamiltonian, but, by~\ref{R43}, not $(1/2+\e)$-resiliently Hamiltonian.
\end{proof}

\subsection{Critical $k$-core expansion}\label{seccrit}

We can now prove Lemma~\ref{coresparseexpand} in the critical case.

\begin{proof}[Proof of Lemma~\ref{coresparseexpand} when $p\leq 4k/n$] Note that we can assume that $\delta \leq 1/8$. Let $k_0$ be sufficiently large that, by Lemma~\ref{mindegexp3}, the following holds in $G(n,p)$ with $p\geq k_0/n$ with probability $1-o(n^{-3})$.
\begin{enumerate}[label = \textbf{T}]
\item There are no two sets $A,B\subset V(G)$ with $|A|\leq C/p$, $|B|\leq 10|A|$ and $e(A,B)\geq \delta pn|A|/2$. \label{early}
\end{enumerate}
Let $k\geq k_0$ and $k/n\leq p\leq 4k/n$. Let $G=G(n,p)$, so that, with probability $1-o(n^{-3})$, $G$ satisfies \ref{early}.

Let $H$ be a spanning $1/2$-residual subgraph of $G^{(k)}$. Note that $\delta pn\leq \delta(4k)\leq k/2$, so that $\delta(H)\geq k/2\geq \delta pn/2$. For each vertex $v\in V(H)$ take $\delta pn/2$ edges incident to $v$, and use these edges to form a subgraph $H_0\subset H$ with $e(H_0)\leq \delta pn^2$. By~\ref{early}, for each $A\subset V(H_0)$ with $|A|\leq C/p$, we have $|N_{H_0}(A)\cup A|\geq 10|A|$, and hence $H_0$ is a $(2,C/p)$-expander.
\end{proof}



\subsection{Supercritical $k$-core expansion}\label{secsuper}
For the resilient expansion of small sets in the $k$-core, we will use the following lemma taken from a result of Krivelevich, Lubetzky and Sudakov~\cite[Claim 3.1]{KLS14}.

\begin{lemma}~\label{coresmallexp} Let $p=c/n$ with $1<c<\log^2 n$ and let $G=G(n,p)$. With probability $1-o(n^{-3})$, any subgraph $H\subset G$ with $\delta(H)\geq 15$ is a $(2,n/2c^2)$-expander. \hfill\qed
\end{lemma}

To show that medium-sized sets resiliently expand in the $k$-core of a typical random graph $G(n,p)$ with $p\geq 4k/n$, we wish to use Lemma~\ref{mindegexp3}. To do so, we will show that it is very likely that most of the vertices in the $k$-core here have at least $pn/4$ neighbours in the $k$-core, as follows.


\begin{lemma}\label{dpncore} There exists $k_0\geq 0$ such that the following holds for any $k\geq k_0$ and $p=c/n$ with $4k\leq c\leq 2\log n$. With probability $1-o(n^{-3})$, in $G=G(n,p)$ there are at most $n/4c^2$ vertices with degree at most $pn/4$ in $G^{(k)}$. 
\end{lemma}
\begin{proof} For each $A\subset V(G)$ with $|A|=n/4c^{2}$, the expected number of edges between $A$ and $V(G)\setminus A$ is
\[
p|A||V(G)\setminus A|\geq 3p|A|n/4=3n/16c.
\]
Therefore, the probability there is some set $A\subset V(G)$ with $|A|=n/4c^{2}$ and $e_G(A,V(G)\setminus A)< |A|\cdot pn/4$ is, using Proposition~\ref{Chern}, at most
\[
\binom{n}{n/4c^{2}}\exp(-3n/(9\cdot 4\cdot 16c))\leq \Big(4ec^2\Big)^{n/4c^{2}}\exp(-n/200c)\leq\exp((3+2\log c-c/50)n/4c^2)=o(n^{-3}),
\]
where $c\geq 4k_0$ is sufficiently large.

Suppose then that $G$ has no such set $A$. Remove vertices with degree at most $pn/4$ iteratively from $G$ until no such vertices exist, or we have removed $n/4c^{2}$ vertices. Let $A$ be the set of removed vertices. Then, noting that $e_G(A,V(G)\setminus A)<|A|\cdot pn/4$, from our supposition we cannot have that $|A|=n/4c^2$. Thus, we must have that $|A|<n/4c^{2}$, and, therefore, as we stopped removing vertices before $n/4c^2$ vertices were removed, $\delta(G-A)> pn/4\geq k$. Hence, $G-A\subset G^{(k)}$ and $V(G)\setminus A$ is a set of at least $(1-1/4c^2)n$ vertices with degree more than $pn/4$ in $G^{(k)}$.
\end{proof}

We can now prove Lemma~\ref{coresparseexpand} in the supercritical case.

\begin{proof}[Proof of Lemma~\ref{coresparseexpand} when $p\geq 4k/n$] Note that we can assume that $\delta<1/8$. Using Lemmas~\ref{mindegexp3} and~\ref{dpncore}, let $k_0\geq 30/\delta$ be sufficiently large that, if $p\geq 4k_0/n$, then the following hold in $G=G(n,p)$ with probability $1-o(n^{-3})$.
\begin{enumerate}[label = \textbf{U\arabic{enumi}}]
\item There are no two sets $A,B\subset V(G)$ with $|A|\leq C/p$, $|B|\leq 10|A|$ and $e(A,B)\geq \delta pn|A|/4$. \label{early0}
\item There are at most $n/4c^2$ vertices in $G^{(k)}$ with degree at most $pn/4$. \label{ear3}
\end{enumerate}
Assume that $G=G(n,p)$ has these properties. By Lemma~\ref{coresmallexp}, with probability $1-o(n^{-3})$, $G$ has the following property.
\begin{enumerate}[label = \textbf{U\arabic*}]\addtocounter{enumi}{2}
\item If $H_0\subset G$ has minimum degree 15, then it is a $(2,n/c^2)$-expander.\label{Rnew32}
\end{enumerate}

Let $H$ be a spanning $1/2$-residual subgraph of $G^{(k)}$. Note that $\delta(H)\geq k/2\geq 15$. For each vertex $v\in V(H)$, take $\min\{d_H(v),\delta pn/2\}\geq 15$ edges incident to $v$, and use these edges to form a subgraph $H_0\subset H$ with $e(H_0)\leq \delta pn^2$.

Let $A\subset V(H_0)$ with $n/2c^2\leq |A|\leq C/p$. By~\ref{ear3}, there are at least $|A|/2$ vertices in $A$ with degree at least $\delta pn/2$ in $H_0$. Thus, by~\ref{early0}, we have $|N_{H_0}(A)\cup A|\geq 10|A|$, so that $|N_{H_0}(A)|\geq 2|A|$. As $\delta(H_0)\geq 15$, by~\ref{Rnew32}, for each $A\subset V(H_0)$ with $|A|\leq n/2c^2$, $|N_{H_0}(A)|\geq 2|A|$.
Thus, for every $A\subset V(H_0)$ with $|A|\leq C/p$, $|N_{H_0}(A)|\geq 2|A|$, and therefore $H_0$ is a $(2,C/p)$-expander.
\end{proof}
This completes the proof of Lemma~\ref{coresparseexpand}, and thus Theorem~\ref{coreres} is proved.

\section*{Acknowledgements}
The author would like to thank two anonymous referees for their comments on this paper.

\bibliographystyle{plain}
\bibliography{rhmreferences}


\end{document}